\ifdefined\ORStyle
  \documentclass[opre,dblanonrev]{informs4}
  \usepackage{eqndefns-left}     
  \RequirePackage{tgtermes}
  \IfFileExists{newtxtext.sty}{\RequirePackage{newtxtext}}{}
  \IfFileExists{newtxmath.sty}{\RequirePackage{newtxmath}}{}
  \RequirePackage{bm}
  \OneAndAHalfSpacedXII
  \EquationsNumberedThrough       
  \TheoremsNumberedThrough        
  \MANUSCRIPTNO{}                 

  \usepackage{natbib}
  \NatBibNumeric
  \bibpunct{[}{]}{,}{n}{}{,}      
  
  \def\bibsep{\smallskipamount}

  \usepackage{booktabs}
  \usepackage{mathtools}
  \usepackage{mathrsfs}
  \usepackage{bbm}      
  \usepackage{or-macros}          
\else
  \documentclass[11pt]{article}
  \usepackage[numbers]{natbib}
  \usepackage{packages}
  \usepackage{editing-macros}
  \usepackage{formatting}
  \usepackage{statistics-macros}

  \usepackage{booktabs}
  \usepackage{mathtools}
  \usepackage{mathrsfs}
  \usepackage[colorlinks=true,allcolors=blue]{hyperref}
\fi

\newcommand{\eif}{\psi}                       
\newcommand{\Veff}{V_{\mathrm{eff}}}          
\newcommand{\Px}{P_{0,X}}                      
\newcommand{\sens}{\mathsf{S}}                

\newcommand{\dd}{\,\mathrm{d}}
\newcommand{\Tot}{\mathcal{T}}                

\newcommand{\Wtwo}{W_2}

\ifdefined\ORStyle
  \RUNAUTHOR{Namkoong}
  \RUNTITLE{Local Sensitivity Under Transport Restrictions}
  \TITLE{Local Sensitivity Under Transport Restrictions}
  \ARTICLEAUTHORS{%
    \AUTHOR{Hongseok Namkoong}
    \AFF{Decision, Risk, and Operations Division, Columbia Business School,
      \EMAIL{namkoong@gsb.columbia.edu}}
  }
  \ABSTRACT{We quantify the value of structural knowledge, restrictions a modeler places
on the world before seeing data. Our analytic workhorse is the \emph{local
  sensitivity} of an estimand to distributional perturbations in the
Otto--Wasserstein geometry: the largest first-order change in the estimand per
unit displacement of probability mass (transport), equal to the dual norm of
the spatial gradient of the efficient influence function.  The modeler encodes
structural knowledge by restricting the class of transports, and the resulting
reduction in sensitivity is the value of their inductive bias.  We illustrate
by shedding light on a longstanding puzzle in causal inference where classical
semiparametric efficiency bounds remain identical regardless of whether the
propensity score is known~\cite{Hahn98}, despite observed practical
difficulties when it is unknown.  Our approach characterizes how known
propensities significantly reduce sensitivity to misspecification.

}
\fi

\begin{document}

\ifdefined\ORStyle
  \maketitle
\else
  \abovedisplayskip=8pt plus0pt minus3pt
  \belowdisplayskip=8pt plus0pt minus3pt

  \begin{center}
    {\huge Local Sensitivity Under Transport Restrictions} \\
    \vspace{.5cm} {\Large Hongseok Namkoong} \\
    \vspace{.2cm}
    {\large Decision, Risk, and Operations Division, Columbia Business School} \\
    \vspace{.2cm}
    \texttt{namkoong@gsb.columbia.edu}
  \end{center}

  \begin{abstract}%
    We quantify the value of structural knowledge, restrictions a modeler places
on the world before seeing data. Our analytic workhorse is the \emph{local
  sensitivity} of an estimand to distributional perturbations in the
Otto--Wasserstein geometry: the largest first-order change in the estimand per
unit displacement of probability mass (transport), equal to the dual norm of
the spatial gradient of the efficient influence function.  The modeler encodes
structural knowledge by restricting the class of transports, and the resulting
reduction in sensitivity is the value of their inductive bias.  We illustrate
by shedding light on a longstanding puzzle in causal inference where classical
semiparametric efficiency bounds remain identical regardless of whether the
propensity score is known~\cite{Hahn98}, despite observed practical
difficulties when it is unknown.  Our approach characterizes how known
propensities significantly reduce sensitivity to misspecification.


  \end{abstract}
\fi


\section{Introduction}
\label{sec:intro}

An estimate or a decision rests on structural knowledge the modeler brings
before seeing data. Such knowledge is a restriction on what the world may look
like, narrowing the laws the modeler deems plausible. An uncertainty set in
robust optimization confines the law to a
neighborhood~\cite{BenTalElGhaouiNemirovski09}, a parametric family or
shrinkage target confines the density to a low-dimensional
manifold~\cite{Stein56, Box76}, and an assumption of no unobserved confounding
rules out the laws in which a hidden cause drives both treatment and
outcome~\cite{RosenbaumRu83}. The same restriction is called an inductive bias
in machine learning~\cite{Mitchell80, Wolpert96}, and we use these terms
interchangeably. We view the value of structural knowledge as the reduction it
produces in the \emph{fundamental difficulty}, a property of the estimand and
the restriction rather than the behavior of any single estimator.

Quantifying how much a particular assumption is worth on a given task has been
a sought-after goal of theorists. The classical theory of semiparametric
efficiency~\citep{LeCam73, Hajek72, BickelKlRiWe93, VanDerVaart98, Newey90,
  Tsiatis07} values a restriction by the reduction it produces in the best
achievable asymptotic variance. Here, local perturbations to the probability
density are represented as scores (reweighting of data), and changes in the
estimand are measured by its functional gradient $\eif$, dubbed the
\emph{efficient influence function}. Perturbations that move the estimand
most, weighed against their visibility in a finite sample, determine the
smallest variance any regular estimator can achieve---which is measured by the
length of $\eif$. Structural knowledge enters by ruling perturbations out, and
the principal device for doing so is the conditional moment restriction.  By
forbidding perturbations in directions that violate this constraint,
semiparametric efficiency theory quantifies the variance reduction achievable
by encoding inductive bias \citep{Chamberlain87}.

However, this classical statistical apparatus falls short in several
respects.  To illustrate, we turn to the discussion in causal inference
around the value of knowing the treatment assignment probabilities across all
covariates (propensity score). Consider an analyst who estimates the average
effect of a treatment from observational records and knows the propensity
score by design. Semiparametric efficiency theory says this knowledge is
worthless: \citet{Hahn98} showed that the asymptotic variance of an efficient
estimator is the same whether the propensity is known or estimated
nonparametrically.  Practitioners have reported otherwise, frequently finding
that a known or well-specified propensity behaves better, and \citet{KangSc07}
make the cautionary case in simulation: even a slightly misspecified
propensity can significantly destabilize (augmented) inverse
probability weighted estimators.  

Instead of considering local reweighting of data, we consider a geometric
displacement of probability mass (``transport'') as the distributional
perturbation and measure the first-order change in the estimand.  We view
model misspecification as a velocity field $v$ on the sample space; the law it
produces at scale $t$ is the pushforward $(\mathrm{Id} + tv)_\# P_0$, and the
first-order change in the estimand pairs $v$ against the spatial gradient
$\nabla\eif$ of the efficient influence function $\eif$. (As we review later,
this pairing follows from the continuity equation.)  Building on a large line
of work that expand a worst-case objective in the radius of a distributional
ball~\citep{Lam16wsc, Lam16, Lam18, GotohKimLim18,
  EsfahaniKu18,VolpiNaSeDuMuSa18, GhoshLam19, BlanchetMu19, DuchiGlNa21,
  BartlDrObWi21, GaoKl22}, we consider the largest first-order change caused
by local displacement of probability mass, per unit transport budget.  This
gives rise to a local sensitivity measure
\begin{equation*}
  \sup_{v \in \Tot,\, \|v\|_{\Tot} \le 1}\, \E_{P_0}[\nabla\eif\cdot v] \;=\;
\|\nabla\eif\|_{\Tot^\star},
\end{equation*}
the dual norm of $\nabla\eif$ over a class of permissible velocities
$\mathcal{T}$. 

A pragmatic advantage of this \emph{geometric measure} of local sensitivity is
that structural assumptions that may be hard to encode as conditional moment
restrictions can be naturally encoded into the class $\mathcal{T}$, and the
resulting geometry turns the qualitative trade-off of the opening into a
quantitative one.  Each such assumption projects $\nabla\eif$ onto a subspace, the
projected dual norm is the sensitivity that survives, and the drop from the
unrestricted gradient energy is what the assumption is worth. In this
framework, restriction on displacements/velocities $v$ is the inductive bias
of the modeler, the projected dual norm is its consequence, and the gap
between the unprojected and projected gradient energies is its value.  As we
illustrate shortly, even for well-studied estimands the local sensitivity
under transport restrictions can shed new light on the role of inductive bias
in ways that classical tools like semiparametric efficiency remain oblivious
to.

However, we do not claim a new mathematical theory. We only look at a
derivative, and the derivative cannot be new; our illustrations are thus
elementary, requiring a velocity field, the continuity equation, and a single
integration by parts.  The first-order change of an estimand under an
infinitesimal transport of the law is the derivative of a performance measure
under an infinitesimal perturbation of the input model, an object studied for
decades in stochastic simulation~\citep{Glasserman91, HoCao91, Glynn90,
  LEcuyer90, ReimanWeiss89, RubinsteinShapiro93, Pflug96, Fu06}. The
particular form we use is a standard object of the Otto--Wasserstein
calculus~\citep{Otto01, Villani09}; \citet{BlanchetMu19, BartlDrObWi21,
  GaoKl22} study sensitivity without significant restrictions on the
perturbations, and we simply highlight how an identical concept can provide
new analytic insights when inductive bias is encoded in the form of transport
restrictions.

We begin by reviewing the well-established Otto--Wasserstein calculus
(Section~\ref{sec:sensitivity}) and quantify the value of inductive bias for a
common estimand in causal inference (Section~\ref{sec:causal}).  In
particular, we show that a known propensity can significantly lower
sensitivity to misspecification, shedding new light on a longstanding
puzzle~\cite{Hahn98, KangSc07}.  Because semiparametric efficiency prices an
assumption only through the sampling variance, an assumption that leaves the
variance unchanged is assigned no value, even when it plainly alters the
estimate's exposure to a misspecified model.


\section{Local sensitivity under transports}
\label{sec:sensitivity}

Let the observation $O \in \mathcal{O}$ be drawn from a law $P_0$, and let
$\theta: \mathcal{P} \to \mathbb{R}$ be a functional on a class of laws,
pathwise differentiable at $P_0$ with efficient influence function
$\eif \in L^2_0(P_0)$: for every smooth submodel $\{P_t\}$ through $P_0$ with
score $h = \partial_t\log p_t|_{t=0}$,
\begin{equation}
  \frac{d}{dt}\theta(P_t)\Big|_{t=0} = \E_{P_0}[\eif(O)\,h(O)] = \langle\eif, h\rangle_{L^2(P_0)}.
  \label{eq:pathwise}
\end{equation}

We are interested in how the estimand $\theta_0 = \theta(P_0)$ responds when
the true law is not $P_0$ but a nearby law obtained by displacing mass through
the sample space.  We define displacement as a velocity field
$v: \mathcal{O} \to \mathbb{R}^d$ on the continuous coordinates, and the law
it produces at scale $t$ is the pushforward
\begin{equation}
  P_t = (\mathrm{Id} + tv)_\# P_0,
  \label{eq:pushforward}
\end{equation}
the law of $O + tv(O)$ when $O \sim P_0$. The map $t \mapsto P_t$ is a curve
through $P_0$ whose velocity at $t = 0$ is $v$, the Otto--Wasserstein tangent
vector at $P_0$ \citep{Otto01, AmbrosioGiSa08, Villani09}.

We rely on the following regularity assumption throughout.
\begin{assumption}[Regularity]
\label{ass:reg}
Write $O = (A, Z)$ with $A$ the discrete coordinates and $Z$ the continuous
ones. Within each stratum $\{A = a\}$, the continuous coordinates have a $C^1$
density positive on an open support, and the efficient influence function
$\eif$ is $C^1$ in them with $\nabla\eif \in L^2(P_0)$. Each velocity field
$v$ displaces the continuous coordinates alone, is $C^1$ in them, induces a
square-integrable score, and carries its mass to zero at the boundary of the
support.
\end{assumption}

\noindent
We ask for regularity within each discrete stratum rather than for the joint
law, because the continuous coordinates can differ from one stratum to the
next.  An outcome recorded only under treatment, for instance, has no density
on the untreated units, so no single support describes the whole sample space.
Conditioning on the discrete coordinates asks for smoothness only where mass
can move.

\subsection{Otto--Wasserstein calculus}
\label{sec:firstorder}

The pushforward path is a smooth submodel, and the continuity equation gives
its score. Within each stratum the density $p$ of $P_t$ solves
$\partial_t p + \nabla\!\cdot\!(p v) = 0$ at first order, so
$\partial_t p|_{t=0} = -\nabla\!\cdot\!(p v)$ and the path carries the score
\begin{equation}
  g_v = \partial_t\log p_t\big|_{t=0} = -\frac{\nabla\!\cdot\!(p v)}{p}.
  \label{eq:gv}
\end{equation}
Under Assumption~\ref{ass:reg}, $g_v \in L^2_0(P_0)$: it integrates to zero by
the divergence theorem and is square-integrable by construction.
By pathwise differentiability~\eqref{eq:pathwise} with
$h = g_v = -\nabla\!\cdot\!(pv)/p$ and integration by parts
\[
  \frac{d}{dt}\theta(P_t)\Big|_{t=0}
  = \langle\eif, g_v\rangle_{L^2(P_0)}
  = \int \eif\,\bigl(-\nabla\!\cdot\!(pv)\bigr)\dd\mu
  = \int \nabla\eif\cdot(pv)\dd\mu
  = \E_{P_0}[\nabla\eif\cdot v],
\]
where we used that $pv \to 0$ at the boundary (Assumption~\ref{ass:reg}). We
record this well-known identity~\cite{Otto01, AmbrosioGiSa08,
  Villani09, Santambrogio15} for later reference.
\begin{proposition}[Transport derivative]
\label{prop:transport-derivative}
Under Assumption~\ref{ass:reg}, the pushforward path~\eqref{eq:pushforward} is a smooth submodel
with score $g_v$ of~\eqref{eq:gv}, and
\begin{equation}
  \frac{d}{dt}\theta(P_t)\Big|_{t=0} = \E_{P_0}[\nabla\eif(O)\cdot v(O)].
  \label{eq:transport-derivative}
\end{equation}
\end{proposition}
\noindent
The first-order change in the estimand pairs the velocity against the spatial
gradient $\nabla\eif$ of the efficient influence function. A constant
influence function (across the continuous coordinates) has $\nabla\eif = 0$
and zero sensitivity, so that the estimand does not respond to mass
displacement, whereas a steep influence function has large gradient energy
regardless of the average size of $\eif$.

We define local sensitivity as the worst displacement in a class of
permissible perturbations.
\begin{definition}[Local sensitivity]
\label{def:sensitivity}
A \emph{velocity class} $\Tot$ is a set of vector fields on $\mathcal{O}$
carrying a norm $\|\cdot\|_{\Tot}$; it is the set of local perturbations the
analyst deems permissible, and its norm is the budget that sizes them. The
local sensitivity of $\theta$ over $\Tot$ at $P_0$ is the largest
first-order change in $\theta$ per unit budget of permissible displacement
\begin{equation}
  \sens(\theta; \Tot) \;:=\; \sup_{v \in \Tot,\, \|v\|_{\Tot} \le 1}\,
  \E_{P_0}[\nabla\eif\cdot v]
  \;=\; \|\nabla\eif\|_{\Tot^\star}~~~\mbox{dual norm over $\Tot$.}
  \label{eq:sensitivity}
\end{equation}
\end{definition}
Without any restrictions on the possible displacements so that $\Tot$ equals
all $L^2(P_0)$ velocity fields, we have
$\sens(\theta; \Tot) = \sqrt{\E_{P_0}\|\nabla\eif(O)\|^2}$ and the supremum
is attained at $v^\star \propto \nabla\eif$.

\subsection{Permissible perturbations as inductive bias}
\label{sec:restrictions-bias}

Our local sensitivity measure~\eqref{eq:sensitivity} allows the modeler to
encode structural knowledge of which local perturbations are permissible. In
particular, we can explicitly write it as a projection of $\nabla \eif$.
\begin{proposition}[Inductive bias as a velocity restriction]
\label{prop:inductive-bias}
Let $\Tot$ be a closed convex cone of velocity fields in $L^2(P_0; \mathbb{R}^d)$ with the native
norm. Then
\begin{equation}
  \sens(\theta; \Tot)
  = \sup_{v \in \Tot,\, \|v\|_{L^2(P_0)} \le 1} \E_{P_0}[\nabla\eif\cdot v]
  = \|\Pi_{\Tot}\nabla\eif\|_{L^2(P_0)}
  \;\le\; \sqrt{\E_{P_0}\|\nabla\eif\|^2},
  \label{eq:restriction-drop}
\end{equation}
where $\Pi_{\Tot}$ is the metric projection onto the cone $\Tot$, the orthogonal projection when
$\Tot$ is a subspace. The inequality is strict whenever $\nabla\eif$ has a nonzero component
outside $\Tot$ on a set of positive measure.
\end{proposition}

\begin{proof}
The Moreau decomposition of a closed convex cone writes $\nabla\eif = \Pi_{\Tot}\nabla\eif +
\Pi_{\Tot^\circ}\nabla\eif$ with the two pieces orthogonal and $\Pi_{\Tot^\circ}\nabla\eif$ in the
polar cone $\Tot^\circ$, so $\langle\nabla\eif, \Pi_{\Tot}\nabla\eif\rangle =
\|\Pi_{\Tot}\nabla\eif\|^2$. For any $v \in \Tot$ with $\|v\|_{L^2(P_0)} \le 1$,
\[
  \langle\nabla\eif, v\rangle = \langle\Pi_{\Tot}\nabla\eif, v\rangle +
  \langle\Pi_{\Tot^\circ}\nabla\eif, v\rangle \le \langle\Pi_{\Tot}\nabla\eif,
  v\rangle \le \|\Pi_{\Tot}\nabla\eif\|_{L^2(P_0)},
\]
the first inequality because
$\langle\Pi_{\Tot^\circ}\nabla\eif, v\rangle \le 0$ for $v \in \Tot$, the
second by Cauchy--Schwarz; equality holds at
$v = \Pi_{\Tot}\nabla\eif/\|\Pi_{\Tot}\nabla\eif\|$.  The gap to the
unrestricted energy is $\|\Pi_{\Tot^\circ}\nabla\eif\|^2$, positive whenever
$\nabla\eif$ leaves $\Tot$ on a set of positive measure.
\end{proof}

The gap $\E_{P_0}\|\nabla\eif\|^2 - \|\Pi_{\Tot}\nabla\eif\|^2_{L^2(P_0)}$ is
the part of the gradient energy the structural assumption rules out, the value
of the modeler's inductive bias. The shape of the permissible velocity class
$\Tot$ decides this gap, not its size alone.
\begin{itemize}\itemsep0pt
  \item A \emph{magnitude budget} on $\|v\|_{L^2(P_0)}$ alone is the unrestricted Wasserstein
  ball; its sensitivity is the unrestricted gradient energy.
\item A \emph{subspace restriction} $\Tot = V$ asserts that displacements lie
  in a specified $V \subset L^2(P_0; \mathbb{R}^d)$; its sensitivity is
  $\|\Pi_V \nabla\eif\|_{L^2(P_0)}$.  Examples include the known-propensity
  restriction $V = \{v : v_x(x)\cdot\nabla\pi_0(x) = 0 \text{ a.s.}\}$ in
  Section~\ref{sec:known-prop} and the confounding restriction
  $V = \{(0, v_y)\}$ in Section~\ref{sec:confounding}.
  \item A \emph{cone restriction} $\Tot = \{v : C(v) \ge 0\}$ encodes a one-sided belief such as
  monotone confounding; its sensitivity is the support function of $\nabla\eif$ over the
  cone.
  \item A \emph{product structure} $\Tot = \{v(o) = \oplus_i v_i(o_i)\}$ asserts that the sample
  space factors; its sensitivity decomposes additively across factors.
  \item A \emph{weighted budget} $\E_{P_0}[v^\top W(o)\,v] \le 1$ with a coordinate-dependent metric
  $W(o) \succ 0$ prices displacement differently across directions. This budget is an ellipsoid
  rather than a cone, so it falls outside Proposition~\ref{prop:inductive-bias}; its sensitivity
  comes from rescaling the metric rather than projecting $\nabla\eif$, and is the dual-weighted
  gradient energy $\sqrt{\E_{P_0}[\nabla\eif^\top W^{-1}\nabla\eif]}$, attained at
  $v^\star \propto W^{-1}\nabla\eif$. Charging displacement heavily along a coordinate encodes a
  smoothness prior, that plausible misspecification cannot move the outcome much along that
  direction, and shrinks the matching component of $\nabla\eif$.
\end{itemize}

\paragraph{Relation to Wasserstein DRO}
The local sensitivity~\eqref{eq:sensitivity} is the first-order slope of a
Wasserstein distributionally robust bound. The pushforward
$(\mathrm{Id} + \delta v)_\# P_0$ sits within $W_2$-distance
$\delta\|v\|_{L^2(P_0)}$ of $P_0$, and by Brenier's theorem every law in a
small $W_2$-ball is, to first order, such a pushforward along the optimal
transport map \citep{Villani09, Santambrogio15}.  When a functional Taylor
expansion holds uniformly over the ball, the Wasserstein DRO radius expansion
gives~\cite{EsfahaniKu18, VolpiNaSeDuMuSa18, BlanchetMu19, BartlDrObWi21,
  GaoKl22}
\begin{equation}
  \sup_{Q:\, \Wtwo(Q, P_0) \le \delta} \bigl|\theta(Q) - \theta(P_0)\bigr|
  = \delta\,\sqrt{\E_{P_0}\|\nabla\eif\|^2} + o(\delta).
  \label{eq:dro-firstorder}
\end{equation}
These robust analyses take the worst case over every distribution within a
radius, with no structure on which departures are plausible, so restricting
the velocity class plays the role of specifying the uncertainty set. A
velocity field is a natural place to encode an inductive bias, since a belief
about how the world can depart from the model becomes a restriction on $v$,
and the projected gradient returns the sensitivity in closed form.

Both Wasserstein DRO and our local sensitivity share a key limitation that is
worth explicit mention: displacement of probability mass can only happen
through continuous coordinates, a purely discrete observation carries no
velocity unless explicitly embedded in a continuous space.

\paragraph{Relation to semiparametric efficiency}
The pairing identity $\langle\nabla\eif, v\rangle = \langle\eif, g_v\rangle$
makes $\nabla\eif$ appear in efficiency bounds:
\begin{equation}
  \sup_{v:\, g_v \neq 0}\,\frac{\langle\nabla\eif, v\rangle^2_{L^2(P_0)}}{\|g_v\|^2_{L^2(P_0)}}
  = \sup_{g \in M,\, g \neq 0}\,\frac{\langle\eif, g\rangle^2_{L^2(P_0)}}{\|g\|^2_{L^2(P_0)}}
  = \bigl\|\Pi_M \eif\bigr\|^2_{L^2(P_0)},
  \label{eq:efficiency-bound}
\end{equation}
where $M = \overline{\Phi(\Tot)}$ collects the scores transport can induce
through $\Phi: v \mapsto g_v = -\nabla\!\cdot\!(pv)/p$; the worst case is the
projection of $\eif$ onto them, at most the H\'ajek--Le Cam variance and equal
to it when every coordinate is continuous.

Local sensitivity~\eqref{eq:sensitivity} sizes a perturbation by the transport
distance $\|v\|_{L^2(P_0)}$, in contrast to semiparametric efficiency theory
which uses the Fisher length of its score, $\|g_v\|_{L^2(P_0)}$, the
information a sample spends to detect it. The two lead to different values
since $g_v$ carries the \emph{divergence of $v$} while $\|v\|$ does not; the
gradient energy can diverge while the efficient variance stays finite, meaning
that an estimand can be easy to estimate yet fragile to misspecification.

\paragraph{Relation to sensitivity analysis in causal inference}
Similar to the contrast between local vs. global sensitivity in stochastic
simulation that summarizes the response over a whole perturbation class by its
variance or its extremes \citep{Sobol01, SaltelliEtAl08}, our approach
contrasts with the partial identification / (global) sensitivity analysis
literature in causal inference which bounds an identified region when a
structural assumption fails by a bounded amount.  Here, analogous to
semiparametric efficiency that introduces local perturbations to the
probability density under a reweighting budget, several authors restrict how a
law is tilted under unobserved confounding~\citep{Rosenbaum87, Rosenbaum02,
  Tan06}.

The choice of \emph{which} norm sizes a perturbation separates the geometric
restrictions developed here from the reweighting restrictions of the classical
sensitivity analysis literature. The continuity-equation map
$\Phi: v \mapsto g_v = -\nabla\!\cdot\!(pv)/p$~\eqref{eq:gv} turns a
displacement into the local reweighting it induces, and the two sides of the
map carry different norms of the same perturbation.  A bound on $g_v$
constrains the local creation of mass (how much the density is reweighted); a
bound on $v$ constrains the spatial displacement (how far mass moves) and is
the natural object for a Lipschitz or physical-magnitude assumption.


\section{A case study in causal inference}
\label{sec:causal}

To demonstrate how a modeler can analyze local sensitivity across a range of
different model misspecifications, we first study a standard estimand in
causal inference, the mean response under selection bias. Formally, we
consider observations $O = (X, A, Y)$ with covariate $X$, selection indicator
$A \in \{0,1\}$, and response $Y = AY(1)$. We focus on the mean potential
outcome $\theta_0 = \E[Y(1)]$ for simplicity; the treatment-effect estimand
$\E[Y(1) - Y(0)]$ contributes a symmetric control-arm term to each
influence-function gradient, and every insight below has an immediate analogue
there.

We assume standard identification conditions: i) unconfoundedness
$Y(1) \perp A \mid X$ and ii) overlap
$\pi_0(x) := P_0(A = 1 \mid X = x) \ge \eta > 0$. The target is
\[
  \theta_0 = \E_{P_0}[Y(1)] = \E_{\Px}[\mu_1(X)]
  ~~~\mbox{where}~~~\mu_1(x) = \E_{P_0}[Y \mid X = x, A = 1].
\]
It is well-known that the efficient influence function of $\theta_0$ is given
by the augmented inverse-propensity-weighted form \citep{Hahn98, RobinsRoZh94}
\begin{equation}
  \eif(o) = \mu_1(x) - \theta_0 + \frac{a}{\pi_0(x)}\bigl(y - \mu_1(x)\bigr).
  \label{eq:eif-aipw}
\end{equation}
The classical semiparametric efficiency bound gives the asymptotic variance
\begin{equation*}
  \Veff = \var\left(\mu_1(X)\right) + \E\left[\frac{\sigma_1^2(X)}{\pi_0(X)}\right]
  ~~~\mbox{where}~~~\sigma_1^2(X) := \var(Y \mid X, A = 1).
\end{equation*}
The efficient variance depends on the zero-th order values of the nuisances.

\begin{figure}[t]
  \centering
  \includegraphics[width=0.95\textwidth]{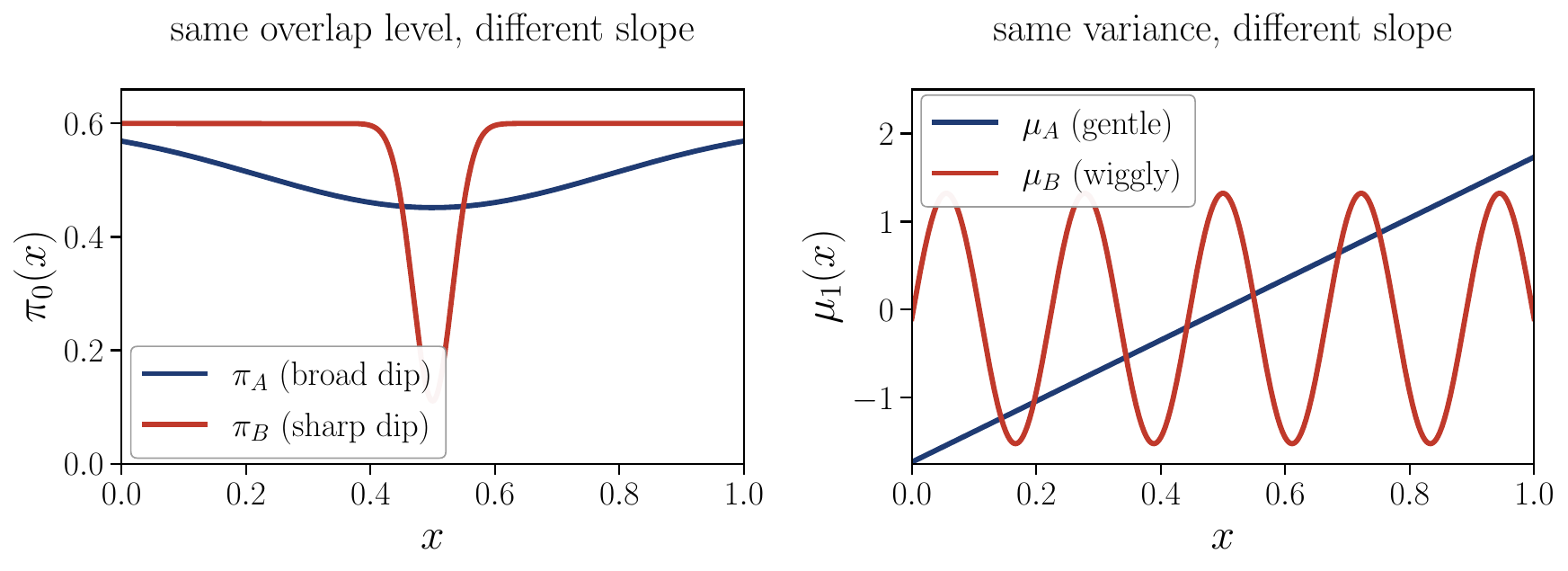}
  \caption{Same level, different slope. \emph{Left:} two propensities with
    identical overlap level $\E[1/\pi_0] = 2.00$, hence identical confounding
    sensitivity, but propensity-slope terms $\E[\|\nabla\pi_0\|^2/\pi_0^3]$ of
    $0.47$ and $455$, nearly a thousandfold difference in covariate-shift
    fragility driven by how steeply overlap degrades, not how low it gets.
    \emph{Right:} two regressions with identical $\mathrm{Var}(\mu_1) = 1$,
    hence identical contribution to the efficient variance, but gradient
    energies $\E\|\nabla\mu_1\|^2$ of $12$ and $808$. A wiggly regression with
    small variance is far more exposed to covariate shift than a monotone one
    with the same variance.}
  \label{fig:slope}
\end{figure}

On the other hand, the local sensitivity carries derivatives of nuisances.
The selection indicator $A$ is discrete and not transported; the continuous
coordinate is $x$ on every unit, and because $Y = A\,Y(1)$ the response $y$
carries a density only on the treated, so Assumption~\ref{ass:reg} applies per
arm as discussed there.
Taking the spatial gradient of the efficient influence function~\eqref{eq:eif-aipw}, we get
\begin{subequations}
  \label{eq:grad-mean}
  \begin{align}
  \partial_y\eif(o) &= \frac{a}{\pi_0(x)}, \label{eq:grad-y}\\
  \partial_x\eif(o) &= \nabla\mu_1(x)\Bigl(1 - \frac{a}{\pi_0(x)}\Bigr)
  - \frac{a\,(y - \mu_1(x))}{\pi_0(x)^2}\,\nabla\pi_0(x). \label{eq:grad-x}
\end{align}
\end{subequations}
Without restrictions, we have the following local sensitivity.
\begin{equation}
  \E_{P_0}\|\nabla\eif\|^2
  = \underbrace{\E_{P_0}\!\Bigl[\bigl(\tfrac{1}{\pi_0(X)} - 1\bigr)\|\nabla\mu_1(X)\|^2\Bigr]}_{\text{regression slope}}
  + \underbrace{\E_{P_0}\!\Bigl[\tfrac{\sigma_1(X)^2\,\|\nabla\pi_0(X)\|^2}{\pi_0(X)^3}\Bigr]}_{\text{propensity slope}}
  + \underbrace{\E_{P_0}\!\Bigl[\tfrac{1}{\pi_0(X)}\Bigr]}_{\text{outcome coordinate}}.
  \label{eq:grad-energy-aipw}
\end{equation}

To gain intuition, consider perturbations that displace the covariate alone,
$\Tot_{\mathrm{cov}} = \{v = (v_x, 0)\}$. This gives the covariate block
$\E_{P_0}\!\Bigl[\bigl(\tfrac{1}{\pi_0} - 1\bigr)\|\nabla\mu_1\|^2\Bigr] +
\E_{P_0}\!\Bigl[\tfrac{\sigma_1^2\,\|\nabla\pi_0\|^2}{\pi_0^3}\Bigr]$ of the
gradient energy~\eqref{eq:grad-energy-aipw}.  Both carry the nuisance
gradients $\nabla\mu_1$ and $\nabla\pi_0$, which the semiparametric efficiency
bound never sees: $\Veff$ reads only the levels $\mathrm{Var}(\mu_1)$ and
$\E[\sigma_1^2/\pi_0]$.

This tells us that the estimand is fragile to covariate shift in proportion to
how fast the nuisances move across the sample space, a rate the level-based
diagnostics~\cite{CrumpHoImMi09} that read the zero-th order value of $\pi_0$
alone cannot see. Two designs identical at the level of the nuisances can
therefore differ sharply in covariate-shift fragility, driven by how steeply
the nuisances move. Figure~\ref{fig:slope} fixes each level and varies the
slope: two propensities with the same $\E[1/\pi_0] = 2.00$ whose
propensity-slope terms differ by nearly a thousandfold, and two regressions
with the same $\mathrm{Var}(\mu_1) = 1$ whose gradient energies
$\E\|\nabla\mu_1\|^2$ differ by a factor of nearly seventy. (Here,
$X \sim \mathrm{Unif}[0,1]$, $\pi_0(x) = 0.6 - d\,e^{-((x - 0.5)/w)^2}$ with
$(w, d) = (0.40, 0.148)$ for the broad dip and $(0.045, 0.488)$ for the sharp
one, each depth chosen so $\E[1/\pi_0] = 2$; the regressions are
$\mu_1(x) \propto x$ and $\mu_1(x) \propto \sin(9\pi x)$, each standardized to
$\mathrm{Var}(\mu_1) = 1$.)

In the following, we illustrate how restrictions on the velocity field can be
used to characterize the value of a correctly specified propensity, and
provide insights on when the estimand is sensitive to unobserved confounding.

\subsection{Knowing the propensity lowers the sensitivity}
\label{sec:known-prop}

We use the local sensitivity framework to study a longstanding gap between
theory and practice in causal inference.  \citet{Hahn98} showed that the
efficient influence function~\eqref{eq:eif-aipw} is orthogonal to the
propensity-score subspace, so the efficient variance is the same under known
and unknown propensity, i.e., knowing $\pi_0$ buys nothing in asymptotic
variance. Working analysts report the opposite anecdotally, that a known or
well-specified propensity behaves better, and tend to dismiss the efficiency
result as a formal curiosity. \citet{KangSc07} confirm this anecdotal evidence
in a controlled simulation setting: under an even slightly misspecified
propensity, inverting the near-zero fitted weights in regions of thin overlap
destabilizes the inverse probability weighted and augmented inverse
probability weighted estimators.

We show that what the practitioner feels when a good propensity model helps is
a reduction in exposure to misspecification, not in sampling variance.  The
invariance of the efficient variance is an algebraic fact about the residual
structure of $\eif$, tied to $\E[Y - \mu_1 \mid X, A = 1] = 0$, and it does
not extend to the geometry of displacements, where the sensitivity reads
derivatives of $\pi_0$ that the variance never sees. Although we focus on the
mean under selection bias for simplicity of exposition, this same pattern
recurs for any doubly-robust functional whose influence function is orthogonal
to the nuisance score in the sense of \citet{RobinsRoZh94, Hahn98}: the
variance is insensitive to knowing the nuisance, the sensitivity is not.

We characterize the reduction in sensitivity a known propensity score can
provide.  Knowing the propensity $\pi_0$ as a function on $\mathcal{X}$ is the
belief that the permissible displacements preserve it. A covariate
displacement advects the propensity along its gradient,
$\pi_0(x + t v_x) = \pi_0(x) + t\,v_x(x)\cdot\nabla\pi_0(x) + o(t)$, so
preserving $\pi_0$ pointwise requires the covariate velocity to be tangent to
the level set of $\pi_0$ at every $x$,
\[
  \Tot^{\mathrm{known}\text{-}\pi} = \bigl\{v = (v_x, v_y) : v_x(x)\cdot\nabla\pi_0(x) = 0 \ \text{a.s.}\bigr\}.
\]
This is a pointwise constraint, not the weaker global orthogonality
$\E[v_x\cdot\nabla\pi_0] = 0$, which would permit a displacement that raises
$\pi_0$ in one region and lowers it in another.

Computing the local sensitivity over this class amounts to projecting
$\nabla\eif$ onto $\Tot^{\mathrm{known}\text{-}\pi}$. Because the class
constrains only the covariate velocity, the projection leaves the outcome
coordinate $\partial_y\eif$ untouched and acts pointwise on the covariate
gradient $\partial_x\eif$, sending it to the hyperplane orthogonal to
$\nabla\pi_0(x)$. The propensity-slope term of~\eqref{eq:grad-x} points along
$\nabla\pi_0$, so the projection annihilates it, while the regression-slope
term keeps only its component $\Pi_{\nabla\pi_0^\perp}\nabla\mu_1$ tangent to
the level set. Knowing the propensity thus erases the propensity slope and
trims the regression slope to its tangential part, leaving the outcome
coordinate alone, giving the known-propensity sensitivity 
\begin{equation}
  \sens\bigl(\theta; \Tot^{\mathrm{known}\text{-}\pi}\bigr)^2
  = \E_{P_0}\!\Bigl[\bigl(\tfrac{1}{\pi_0} - 1\bigr)\,\bigl\|\Pi_{\nabla\pi_0^\perp}\nabla\mu_1\bigr\|^2\Bigr]
  + \E_{P_0}\!\Bigl[\tfrac{1}{\pi_0}\Bigr],
  \label{eq:known-pi-sens}
\end{equation}
where $\Pi_{\nabla\pi_0^\perp}$ is the projection that removes the direction
$\nabla\pi_0(x)$ at each $x$. The regression slope enters only through its
component tangent to the propensity level set, reducing to $\|\nabla\mu_1\|^2$
only when $\nabla\mu_1 \perp \nabla\pi_0$ almost surely.

Comparing the unrestricted gradient energy~\eqref{eq:grad-energy-aipw} with
the known-propensity sensitivity~\eqref{eq:known-pi-sens} isolates what
knowing $\pi_0$ is worth. The semiparametric efficiency bound is unchanged
whether $\pi_0$ is known or estimated nonparametrically~\citep{Hahn98},
$\Veff^{\mathrm{known}\text{-}\pi} = \Veff = \mathrm{Var}(\mu_1) +
\E[\sigma_1^2/\pi_0]$, yet once we forbid the displacement from misspecifying
$\pi_0$ the sensitivity strictly drops, by
\begin{equation}
  \sens(\theta; \Tot)^2 - \sens(\theta; \Tot^{\mathrm{known}\text{-}\pi})^2
  = \E_{P_0}\!\Bigl[\bigl(\tfrac{1}{\pi_0} - 1\bigr)\,\bigl\|\Pi_{\nabla\pi_0}\nabla\mu_1\bigr\|^2\Bigr]
  + \E_{P_0}\!\Bigl[\tfrac{\sigma_1^2\,\|\nabla\pi_0\|^2}{\pi_0^3}\Bigr] \;>\; 0.
  \label{eq:hahn-gap}
\end{equation}
Knowing the propensity never increases the sensitivity, and the drop is strictly
positive unless, at almost every covariate value with $\nabla\pi_0 \ne 0$, the
outcome is noiseless ($\sigma_1^2 = 0$) and $\nabla\mu_1$ is orthogonal to
$\nabla\pi_0$.

The value of knowing the propensity has two parts, both carried by the
propensity gradient $\nabla\pi_0$. The propensity-slope term
$\E[\sigma_1^2 \|\nabla\pi_0\|^2/\pi_0^3]$ grows with how steeply assignment
varies; the regression-slope term
$\E[(\tfrac{1}{\pi_0} - 1)\|\Pi_{\nabla\pi_0}\nabla\mu_1\|^2]$ grows instead
with how closely the direction of $\nabla\pi_0$ aligns with $\nabla\mu_1$,
through
$\|\Pi_{\nabla\pi_0}\nabla\mu_1\|^2 =
(\nabla\mu_1\cdot\nabla\pi_0)^2/\|\nabla\pi_0\|^2$, and is invariant to the
magnitude of $\nabla\pi_0$. The value grows with the strength of
confounding-by-indication and vanishes when assignment is unrelated to the
covariates. The randomized trial is an interesting edge case, with a
propensity known to be constant in the covariates, $\nabla\pi_0 = 0$; here
knowing the propensity lowers neither the efficient variance nor the
sensitivity even when the outcome is noisy. This accords with a long-standing
finding in the design of such trials where the precision gained by covariate
adjustment has been observed to come from the association between the
covariates and the outcome rather than from the propensity
\citep{WilliamsonFW14, Zeng21}.

In sum, knowing the propensity improves robustness exactly when it varies,
that is, in the observational settings where the analyst must estimate it in
the first place.

\subsection{Impact of unobserved confounding concentrates where overlap is
  thin}
\label{sec:confounding}

We now turn our attention to unobserved confounders that lead to violations in
the ignorability assumption $Y(1) \perp A \mid X$. Classical global
sensitivity analysis bounds how much a hidden variable can reweight the
propensity or the outcome density and reports a partial identification
interval on the estimand~\citep{Rosenbaum87}. Under our geometric framework,
we can model the impact of unobserved confounding as a displacement of the
treated outcome.  This lets us see what the worst confounder looks like as an
explicit function of the covariate, i.e., we can learn where in the sample
space confounding does its damage, not only how large that damage can be.  Any
belief about how far the outcome can move along a coordinate can also be
encoded as a restriction on the velocity.

We demonstrate the flexibility of our modeling apparatus by illustrating how
local changes \emph{under} the ignorability assumption can be viewed as
sensitivity to confounding. Our first observation is that unobserved
confounding impacts inferential results when it silently changes the
distribution of $Y|X, A=1$; because we observe $Y = A\,Y(1)$, the response
coordinate carries information about $Y(1)$ only on the treated.  A (local)
confounder is therefore a velocity on the response coordinate supported on the
treated, and the $L^2(P_0)$ energy of such a field $v = (0, v_y)$ is exactly
$\E_{P_0}[\|v\|^2] = \E_{P_0}[A\,v_y^2]$. Taking this energy as the budget
gives the velocity class
\begin{equation}
  \Tot_{\mathrm{conf}} = \bigl\{v = (0, v_y) : \operatorname{supp} v_y \subseteq \{a = 1\}\bigr\},
  \qquad \text{budget } \E_{P_0}[A\,v_y^2] \le 1.
  \label{eq:conf-class}
\end{equation}
The treatment indicator $A$ multiplying $v_y$ is not a modeling choice but a
consequence of $Y = A\,Y(1)$: a displacement of the treated potential outcome
registers in the observed law only where $A = 1$, so the structure of the
observation forces both the support of $v_y$ and the weight in the budget.

\begin{figure}[t]
  \centering
  \includegraphics[width=0.62\textwidth]{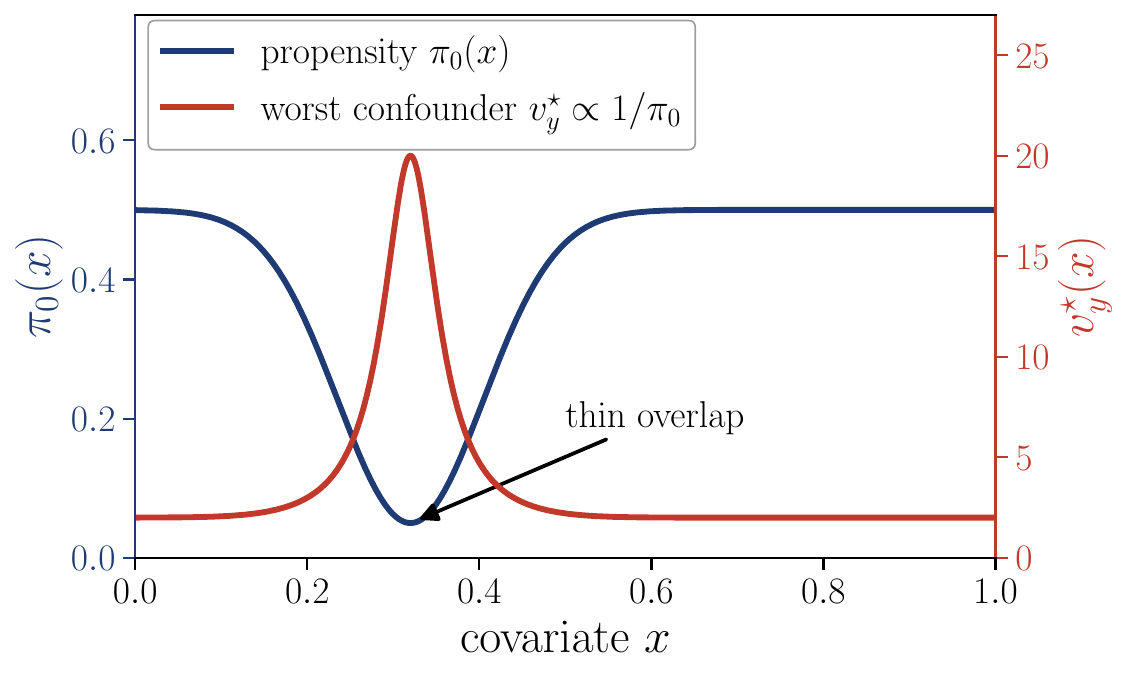}
  \caption{The worst-case unobserved confounder $v_y^\star \propto 1/\pi_0$
    (red) places its distortion where overlap $\pi_0$ (blue) is thin. A
    confounder of the same budget but spread evenly across the population
    moves the estimand far less.}
  \label{fig:conf-overlap}
\end{figure}

Pairing $v$ against the outcome gradient~\eqref{eq:grad-y}, the first-order
change~\eqref{eq:transport-derivative} is $\E_{P_0}[(A/\pi_0)\,v_y]$. Since
$v_y$ is supported on the treated, conditioning on $X$ and using
$\E_{P_0}[A\,v_y \mid X] = \pi_0(X)\,\E_{P_0}[v_y \mid X, A = 1]$ cancels the
propensity,
\begin{equation}
  \frac{d}{dt}\theta(P_t)\Big|_{t=0}
  = \E_{P_0}\!\Bigl[\frac{A}{\pi_0(X)}\,v_y\Bigr]
  = \E_{\Px}\bigl[\E[v_y \mid X, A = 1]\bigr],
  \label{eq:conf-change}
\end{equation}
the average over covariates of the mean treated-arm distortion at each $x$. A
distortion of fixed conditional size moves the estimand by the same amount at
every $x$, while the budget charges it less where the treated are scarce.

Maximizing~\eqref{eq:conf-change} over the budget in~\eqref{eq:conf-class} gives
\begin{equation}
  \sens(\theta; \Tot_{\mathrm{conf}}) = \sqrt{\E_{P_0}[1/\pi_0]},
  \qquad v_y^\star(x) \;\propto\; \frac{1}{\pi_0(x)} \ \text{ on the treated.}
  \label{eq:conf-worst}
\end{equation}
The worst confounder is a function of the covariate: it is inversely
proportional to the propensity $\pi_0(x)$, so it loads its distortion onto the
covariate region where the treated are scarce. Figure~\ref{fig:conf-overlap}
draws $v_y^\star(x)$ against a propensity with a low-overlap region.
Confounding does the most damage where the data are thinnest.

The amplification of the worst confounder over an evenly spread one measures
how heterogeneous overlap is. A confounder spread evenly over the treated,
$v_y \equiv \text{const}$, moves the estimand by $1/\sqrt{P_0(A=1)}$ per unit
budget, so the ratio of the worst-case to the even confounder is
\begin{equation}
  \frac{\sqrt{\E_{P_0}[1/\pi_0]}}{1/\sqrt{P_0(A=1)}} = \sqrt{\E_{P_0}[1/\pi_0]\,\E_{P_0}[\pi_0]} \;\ge\; 1,
  \label{eq:conf-amp}
\end{equation}
which grows with the spread of $\pi_0$. A confounder is dangerous in
proportion to how it aligns with overlap, not how large it is on average. A
confounder that loads on the well-covered majority is nearly harmless, and one
that loads on the thin-overlap minority is amplified by the overlap
heterogeneity.

The displacement budget also makes structural knowledge about the outcome easy
to add. Write $P_0^{(1)}$ for the treated sub-probability measure so that the
budget is the squared norm $\E_{P_0}[A\,v_y^2] = \|v_y\|_{L^2(P_0^{(1)})}^2$
and the first-order change~\eqref{eq:conf-change} pairs $v_y$ against the
representer $1/\pi_0 \in L^2(P_0^{(1)})$. Suppose the analyst believes the
hidden variable cannot distort the treated outcome differently across a
covariate direction $x_e$, a smoothness prior on how confounding enters the
outcome. This confines the distortion to the subspace
$V_e = \{v_y : \partial_{x_e} v_y = 0\}$ of fields constant along $x_e$, and
the worst-case confounder becomes the $L^2(P_0^{(1)})$ projection of $1/\pi_0$
onto $V_e$,
\begin{equation}
  \sens(\theta; \Tot_{\mathrm{conf}} \cap V_e)
  = \bigl\|\Pi_{V_e}(1/\pi_0)\bigr\|_{L^2(P_0^{(1)})}
  \;\le\; \sqrt{\E_{\Px}[1/\pi_0]},
  \label{eq:conf-rigid}
\end{equation}
the gap to the unrestricted sensitivity~\eqref{eq:conf-worst} being the
confounding this knowledge rules out
(Proposition~\ref{prop:inductive-bias}). A belief that the outcome cannot move
much along a coordinate is awkward to phrase as a bound on how the confounder
reweights the density, but in the transport geometry it is a restriction on
the velocity, and the projected representer returns the sensitivity directly.

The region exposed to confounding and the region hard to estimate need not
coincide. The efficient variance weights inverse overlap by the outcome noise,
$\E[\sigma_1^2/\pi_0]$, while the confounding sensitivity weights it by
nothing, $\E[1/\pi_0]$, so a thin-overlap region with quiet outcomes is
fragile to confounding but cheap to estimate, and a well-covered noisy region
is the reverse.


\section{Discussion}
\label{sec:discussion}


We introduce local sensitivity under transport restrictions as a vehicle in
which an inductive bias can be encoded and priced. The local
sensitivity of an estimand is the dual norm of the influence-function gradient
$\nabla\eif$ over a class of permissible velocity fields, and structural
knowledge enters through the choice of this class: an analyst who admits only
displacements that preserve the propensity, touch the response coordinate
alone, follow a fixed direction, or respect a smoothness prior thereby lowers
the sensitivity by the projection onto the directions ruled out
(Proposition~\ref{prop:inductive-bias}).  The restriction is the inductive
bias, the projected dual norm is its consequence, and the gap between the
unprojected and projected gradient energies measures what the assumption is
worth. Because a velocity field lives on the sample space, the beliefs it
expresses are geometric: a magnitude bound is a smoothness or Lipschitz prior
on how far the world can move, a subspace is a known direction of departure, a
cone is a monotone restriction, and a product structure is a decoupling. Such
beliefs are awkward to phrase as reweightings of the density but immediate as
velocity classes.

We reiterate that the local sensitivity measure we study is itself standard:
the first-order change of an estimand under a transport of the law is an
standard object of the Otto--Wasserstein calculus \citep{Otto01, Villani09}
and it cannot be new. In particular, its worst case over a Wasserstein ball is
the first-order coefficient computed in the distributionally robust
optimization literature \citep{BlanchetMu19, BartlDrObWi21, GaoKl22}: the
unrestricted sensitivity $\sqrt{\E_{P_0}\|\nabla\eif\|^2}$ is exactly this
coefficient by the expansion~\eqref{eq:dro-firstorder}. We take this as given,
and propose the restricted version as a modeling device. The construction is
confined to continuous coordinates, since a velocity displaces mass through a
continuous sample space and a purely discrete observation admits none; even
so, we believe transport restrictions offer a useful alternative to
reweighting, one that makes smoothness and geometric structural knowledge
straightforward to state and to price, and we expect the questions they raise
to repay further study.



\ifdefined\ORStyle
  \bibliographystyle{informs2014}
\else
  \bibliographystyle{abbrvnat}
  \setlength{\bibsep}{.7em}
\fi
\bibliography{refs}

\end{document}